\documentclass[a4paper,11pt]{article}

\usepackage{amsmath}
\usepackage{amsthm}
\usepackage{amssymb}
\usepackage{cancel}
\usepackage{color}
\usepackage{calrsfs}
\usepackage{varioref}
\usepackage{bbm}
\definecolor{dblue}{rgb}{0.09,0.32,0.44} 
\usepackage[pdfborder={0 0 0}, colorlinks=true, pdfborderstyle={}, linkcolor=dblue, citecolor=dblue, urlcolor=blue]{hyperref}
\usepackage{epsfig}
\usepackage{psfrag}
\usepackage{subfigure}
\usepackage{graphicx}
\usepackage[mathscr,mathcal]{eucal}
\usepackage[shortlabels]{enumitem}
\usepackage{t1enc}
\usepackage[latin2]{inputenc}

\newtheorem {theorem}{Theorem}
\newtheorem {lemma}{Lemma}
\newtheorem {corollary}{Corollary}
\newtheorem {proposition}{Proposition}

\newtheorem* {theorem*}{Theorem}
\newtheorem* {lemma*}{Lemma}
\newtheorem* {corollary*}{Corollary}
\newtheorem* {proposition*}{Proposition}
\newtheorem* {definition*}{Definition}
\newtheorem* {conjecture*}{Conjecture}
\newtheorem* {question*}{Question}

\newtheorem* {theoremkv*} {Theorem KV}
\newtheorem* {corollarykv*} {Corollary KV}
\newtheorem* {theoremrsc1*} {Theorem RSC1}
\newtheorem* {theoremrsc2*} {Theorem RSC2}

\theoremstyle{remark}
\newtheorem* {remark*}{Remark}

\def \R {\mathbb R}

\def \Z {\mathbb Z}

\def\cE{\mathcal{E}}
\def\cF{\mathcal{F}}
\def\cG{\mathcal{G}}
\def\cH{\mathcal{H}}

\def\cL{\mathcal{L}}

\def\vareps{\varepsilon}

\newcommand{\probab}[1]{\ensuremath{\mathbf{P}\left(#1\right)}}
\newcommand{\expect}[1]{\ensuremath{\mathbf{E}\left(#1\right)}}

\newcommand{\probabom}[1]{\ensuremath{\mathbf{P}_{\omega}\left(#1\right)}}
\newcommand{\expectom}[1]{\ensuremath{\mathbf{E}_{\omega}\left(#1\right)}}

\newcommand{\condprobabom}[2]{\ensuremath{\mathbf{P}_{\omega}\left(#1\bigm|#2\right)}}
\newcommand{\condexpectom}[2]{\ensuremath{\mathbf{E}_{\omega}\left(#1\bigm|#2\right)}}

\newcommand{\ind}[1]{\ensuremath{\mathbbm{1}_{\{#1\}}}}

\def\clap#1{\hbox to 0pt{\hss#1\hss}}

\def\ordo{o}

\def\Ker{\mathrm{Ker}}
\def\Ran{\mathrm{Ran}}
\def\Dom{\mathrm{Dom}}
\def\un{\underline{n}}

\renewcommand{\d}{\mathrm d}

\newcommand{\abs}[1]{\ensuremath\left|{#1}\right|}
\newcommand{\norm}[1]{\ensuremath\left\|{#1}\right\|}

\def\wh{\widehat}

\begin{document}

\title{Quenched Central Limit Theorem for Random Walks in Doubly Stochastic Random Environment}

\author{
{\sc B\'alint T\'oth}
\\[8pt]
{University of Bristol, UK and R\'enyi Institute, Budapest, HU}
}

\maketitle

\begin{abstract}
\noindent
We prove the quenched version of the central limit theorem for the displacement of a random walk in doubly stochastic random environment, under the $H_{-1}$-condition, with slightly stronger,  $\cL^{2+\varepsilon}$ (rather than $\cL^2$) integrability condition on the stream tensor. On the way we extend Nash's moment bound to the non-reversible, divergence-free drift case.  

\medskip\noindent
{\sc MSC2010: 60F05, 60G99, 60K37}

\medskip\noindent
{\sc Key words and phrases:} random walk in random environment, quenched central limit theorem, Nash bounds.

\end{abstract}

\section{Introduction}

Let $(\Omega, \cF, \pi, \tau_z:z\in\Z^d)$ be a probability space with an ergodic $\Z^d$-action. Denote by  $\cE:=\{k\in\Z^d: |k|=1\}$ the set of possible steps of a nearest-neighbour walk on $\Z^d$. Let $p_k:\Omega\to[0,s^*]$, $k\in\cE$, be bounded measurable functions  ($s^*<\infty$ is their common upper bound). These will be the jump rates of the RWRE considered (see \eqref{the walk} below) and assume they are \emph{doubly stochastic}, 
\begin{align}
\label{bistoch}
\sum_{k\in\cE}p_k(\omega)
=
\sum_{k\in\cE}p_{-k}(\tau_k\omega).
\end{align}
The physical meaning of \eqref{bistoch} is, that the local drift field of the walk is \emph{divergence-free}, i.e. the stream field of an \emph{incompressible} flow in stationary regime. 

Given these,  define the continuous time nearest neighbour random walk $t\mapsto X(t)\in\Z^d$ as a Markov process on $\Z^d$, with  $X(0)=0$ and conditional jump rates
\begin{align}
\label{the walk}
\condprobabom{X(t+dt)= x+k}{X(t)=x} = p_k(\tau_x\omega) dt + \ordo(dt),
\end{align}
where the subscript $\omega$ denotes that the random walk $X(t)$ is a Markov process on $\Z^d$ \emph{conditionally}, with fixed $\omega\in\Omega$, sampled according to $\pi$. The continuous setup is for convenience only. Since the jump rates are bounded this is fully equivalent with a discrete time walk. 

We will use the notation $\probabom{\cdot}$ and $\expectom{\cdot}$  for \emph{quenched} probability and expectation. That is: probability and  expectation with respect to the distribution of the random walk $X(t)$, \emph{conditionally, with given fixed environment $\omega$}. The notation $\probab{\cdot}:=\int_\Omega\probabom{\cdot} {\d}\pi(\omega)$ and  $\expect{\cdot}:=\int_\Omega\expectom{\cdot} {\d}\pi(\omega)$ will be reserved for \emph{annealed} probability and expectation. That is: probability and  expectation with respect to the random walk trajectory $X(t)$ \emph{and} the environment $\omega$, sampled according to the distribution $\pi$. 

It is well known (and easy to check, see e.g.\ \cite{kozlov_85}) that due to double stochasticity \eqref{bistoch} the annealed set-up is stationary and ergodic in time: the process of the environment as seen from the position of the random walker
\begin{align}
\label{env proc}
\eta(t):=\tau_{X(t)}\omega
\end{align}
is a stationary and ergodic Markov process on $(\Omega, \pi)$ and consequently the random walk $t\mapsto X(t)$ will have stationary and ergodic annealed increments.

Next we define, for $k\in\cE$, $s_k:\Omega\to[0,s^*]$, $v_k:\Omega\to[-s^*,s^*]$, and $\psi,\varphi:\Omega\to\R^d$,
\begin{align}
\label{s v phi psi}
\begin{aligned}
&
s_k(\omega):=\frac{p_k(\omega)+p_{-k}(\tau_k\omega)}{2},
\qquad\qquad
&&
\psi(\omega):= \sum_{k\in\cE} k s_k(\omega), 
\\
&
v_k(\omega):=\frac{p_k(\omega)-p_{-k}(\tau_k\omega)}{2},
\qquad\qquad
&&
\varphi(\omega):= \sum_{k\in\cE} k v_k(\omega).
\end{aligned}
\end{align}
The local \emph{quenched} drift of the random walk is
\begin{align}
\notag
\condexpectom{dX(t)}{X(t)=x} 
= 
\left(\psi(\tau_x\omega)+\varphi(\tau_x\omega)\right) dt +\ordo(dt).
\end{align}
Note that from the definitions \eqref{s v phi psi} it follows that for $\pi$-almost all $\omega\in\Omega$
\begin{align}
\label{symmetries}
\begin{aligned}
&
s_k(\omega)-s_{-k}(\tau_k\omega)=0,
\qquad\qquad
&&
\psi_i(\omega)=s_{e_i}(\omega)-s_{e_i}(\tau_{-e_i}\omega), 
\\
&
v_k(\omega)+v_{-k}(\tau_k\omega)=0,
\qquad\qquad
&&
\varphi_i(\omega)=v_{e_i}(\omega)+v_{e_i}(\tau_{-e_i}\omega).
\end{aligned}
\end{align}
In addition, condition \eqref{bistoch} is equivalent to
\begin{align}
\label{divfree}
\sum_{k\in\cE}v_k(\omega)\equiv0, 
\qquad 
\pi\text{-a.s.}
\end{align}
Thus, $\left(v_k(\tau_x\omega)\right)_{k\in\cE, x\in\Z^d}$ is a stationary sourceless (or, di\-ver\-gence-free) flow on the lattice $\Z^d$. The \emph{physical interpretation} of the divergence-free condition \eqref{divfree} is that the walk \eqref{the walk} models the motion of a particle suspended in stationary, \emph{incompressible flow}, with thermal noise. 

In order that the walk $t\mapsto X(t)$ have \emph{zero annealed mean drift} we also assume that for all $k\in\cE$
\begin{align}
\label{nodrift}
\int_\Omega v_k(\omega) \,{\d}\pi(\omega)
=0.
\end{align}

Our next assumption is the \emph{strong ellipticity} condition for the symmetric part of the jump rates: there exists another positive constant  $s_*\in(0,s^*]$ such that for $\pi$-almost all $\omega\in\Omega$ and all $k\in\cE$
\begin{align}
\label{ellipt}
s_k(\omega)\ge s_*, 
\quad 
\pi\text{-a.s.}
\end{align}
Note that the ellipticity condition is imposed only on the symmetric part $s_k$ of the jump rates and not on the jump rates $p_k$. It may happen that 
$\pi( \{ \omega: p_k(\omega)=0 \} )>0$, as it is the case in some examples given in \cite{kozma-toth_17}. 

By applying a linear time change we may and will choose $s_*=1\le s^*<\infty$.

Finally, we formulate the notorious \emph{$H_{-1}$-condition} which plays a key role. Denote for $i,j=1,\dots,d$, 
\begin{align}
\notag
&
C_{ij}(x)
:=
\int_\Omega \varphi_i(\omega)\varphi_j(\tau_x\omega)d\pi(\omega),
&
\wh C_{ij}(p)
:=
\sum_{x\in\Z^d} e^{\sqrt{-1}x\cdot p} C_{ij}(x).
\end{align}
By Bochner's theorem, the Fourier transform $\wh C$ is  positive definite $d\times d$-matrix-valued-measure on $[-\pi,\pi)^d$. The no-drift condition \eqref{nodrift} is equivalent to $\wh C_{ij}(\{0\})=0$, for all $i,j=1,\dots,d$. With slight abuse of notation we denote this measure formally as $\wh C_{ij}(p)dp$ even though it could be not absolutely continuous with respect to Lebesgue. 

The \emph{$H_{-1}$-condition} is the following: 
\begin{align}
\label{hcond}
\int_{[-\pi,\pi)^d}
\left(\sum_{j=1}^d(1-\cos p_j) \right)^{-1} \sum_{i=1}^d \wh C_{ii}(p) \, {\d}p <\infty.
\end{align}
This is an \emph{infrared bound} on the correlations of the skew-symmetric part of the drift field, $x\mapsto\varphi(\tau_x\omega)\in\R^d$. It implies diffusive upper bound on the annealed walk (see the upper bound in (KT28)) and turns out to be a natural sufficient condition for the diffusive scaling limit (that is, CLT for the annealed walk), see Theorem 1 in \cite{kozma-toth_17}. [Throughout this note (KTxx) points at display number (xx) in \cite{kozma-toth_17}.] Three other equivalent formulations of \eqref{hcond} are given in \cite{kozma-toth_17}. Two of these, (KT35) and Proposition 4(ii) of \cite{kozma-toth_17} are of particular interest, since we shall use them. Note that the $H_{-1}$-condition \eqref{hcond} actually formally implies the no-drift condition \eqref{nodrift}.

It is proved in Proposition 4 (ii) of \cite{kozma-toth_17} that the $H_{-1}$-condition \eqref{hcond} is equivalent to the existence of a stationary and square integrable stream-tensor-field whose curl (or, rotation) is exactly the source-less (divergence-free) flow $v$. More explicitly, there exist  $h_{k,l}\in\cH$, $k,l\in\cE$, with symmetries
\begin{align}
\label{htensor}
h_{k,l}(\omega)
=
-h_{-k,l}(\tau_k\omega)
=
-h_{k,-l}(\tau_l\omega)
=
-h_{l,k}(\omega)
\quad
\pi\text{-a.s},
\end{align}
such that 
\begin{align}
\label{vcurl}
v_k(\omega)=\sum_{l\in\cE}h_{k,l}(\omega).
\end{align}

\noindent{\bf Remarks on the stream tensor $h$.}
The fact that $v$ is expressed as in \eqref{vcurl} with $h$ having the symmetries \eqref{htensor} is essentially the lattice-version of Helmholtz's theorem (in its most common three-dimensional formulation: "a divergence free vector field is the curl of a vector field"). Note that \eqref{htensor} means that the stream tensor field $x\mapsto h(\tau_x\omega)$ is actually function of the \emph{oriented plaquettes} of $\Z^d$. In particular, in two-dimensions $x\mapsto h(\tau_x\omega)$ defines a stationary \emph{height function} on the dual lattice $\Z^{2}+ (1/2,1/2)$, in three-dimensions $x\mapsto h(\tau_x\omega)$ defines a stationary \emph{oriented flow} (that is: a vector field) on the dual lattice $\Z^{3}+ (1/2,1/2,1/2)$. For more details about the stream tensor and the derivation of \eqref{htensor}-\eqref{vcurl} see section 5 of \cite{kozma-toth_17}.

\smallskip

We will now assume that the stream-tensor-field has the stronger integrability 
\begin{align}
\label{two-plus-eps}
h\in\cL^{2+\vareps},
\end{align}
for some $\vareps>0$, rather than being merely square integrable. This stronger integrability condition is needed in the proof of quenched tightness of the diffusively scaled displacement $t^{-1/2}X(t)$. We denote 
\begin{align}
\label{hbound}
h^*
=
h^*(\vareps)
:=
\sum_{k,l\in\cE}\left(\int_\Omega \abs{h_{k,l}}^{2+\vareps} d\pi\right)^{1/(2+\vareps)}
<\infty.
\end{align} 

In \cite{kozma-toth_17} it was shown that for a RWRE \eqref{the walk} whose environment satisfies conditions \eqref{bistoch}, \eqref{ellipt} and \eqref{hcond} the central limit theorem holds, under diffusive scaling and Gaussian limit with finite and nondegenerate asymptotic covariance, \emph{in probability with respect to the environment}, see Theorem 1 in \cite{kozma-toth_17}. The proof is based on the \emph{relaxed sector condition} introduced in \cite{horvath-toth-veto_12} and down-to-earth explicit functional analysis in and over the Hilbert spaces of scalars ($\cH$) and gradients ($\cG$):
\begin{align}
\notag
\cH
&:=
\{ f\in\cL^2(\Omega, \pi): \int_{\Omega} f(\omega)d\pi(\omega)=0\}, 
\\
\notag
\cG
&:=
\{g=(g_k)_{k\in\cE}\in\oplus_{k\in\cE}\cH \ : \ 
\\
\notag
&
\hskip25mm
g_k(\omega)+g_{-k}(\tau_k\omega)=0, \ \
g_k(\omega)+g_l(\tau_k\omega)=g_l(\omega)+g_k(\tau_l\omega), \ \ 
k,l\in\cE\}.
\end{align} 

The main result of the present paper is, that under conditions \eqref{bistoch}, \eqref{ellipt}, \eqref{hcond} and the a marginally stronger integrability condition \eqref{two-plus-eps}version of \eqref{hcond} actually the \emph{quenched CLT} holds, with deterministic nondegenerate covariance matrix. This is Theorem \ref{thm: quenched clt} below.

For general background on RWRE and in particular on the quenched/annealed CLT dichotomy see the surveys \cite{zeitouni_04}, \cite{biskup_11}, \cite{kumagai_14}.  For more background on random walks in doubly stochastic random environment, for interesting examples and in general more illuminating comments  see \cite{kozma-toth_17}.

\section{Results}

Throughout the paper conditions \eqref{bistoch}, \eqref{ellipt} and \eqref{hcond} are assumed. (Recall that \eqref{nodrift} is formally implied by \eqref{hcond}, so we don't state it as a separate condition.) Propositions \ref{prop:harmonic_coordinates} and \ref{prop:martingale clt} are valid under these conditions. In Proposition \ref{prop:tightness}, and as a consequence, in Proposition \ref{prop:small error} and Theorem \ref{thm: quenched clt} the stronger integrability condition \eqref{two-plus-eps} of the stream-tensor-field  is also assumed. 

\begin{proposition}
\label{prop:tightness}
Conditions \eqref{bistoch}, \eqref{ellipt}, \eqref{hcond}, \eqref{two-plus-eps} are assumed. 
There exists a constant $M^*=M^*(d, s_*, s^*, \varepsilon, h^*)<\infty$ such that for $\pi$-almost all $\omega$ 
\begin{align}
\label{tightness}
\varlimsup_{t\to\infty} t^{-1/2} \expectom{\abs{X(t)}} \le M^*.
\end{align}
In particular the scaled displacements $t^{-1/2}X(t)$ are quenched tight. 
\end{proposition}

In the next proposition $\Delta$ denotes the Laplacian operator acting on the Hilbert space $\cH$, as defined in \eqref{grad-lap-riesz} below. Note that $\Delta$ is bounded, self-adjoint and negative. Thus, the operators $\abs{\Delta}^{1/2}$ and  $\abs{\Delta}^{-1/2}$ are defined in terms of the spectral theorem. The unbounded operator $\abs{\Delta}^{-1/2}$ is defined on the domain 
\begin{align}
\notag
\cH_{-1}:=
\{\phi\in\cH: 
\lim_{\lambda\searrow0}
(\phi, (\lambda I -\Delta)^{-1}\phi)_{\cH}<\infty\}
=
\Ran \abs{\Delta}^{1/2}
=
\Dom \abs{\Delta}^{-1/2}.
\end{align}

\begin{proposition}
\label{prop:harmonic_coordinates}
Conditions \eqref{bistoch}, \eqref{ellipt}, \eqref{hcond} are assumed. 
For any $\phi\in\cH_{-1}$
there exists a unique solution $\theta\in\cG$ of the equation
\begin{align}
\label{harm1}
&
\sum_{k\in\cE} p_k(\omega)\theta_k(\omega) = \phi(\omega).
\end{align}
\end{proposition}

\noindent
We denote by $\Theta:\Omega\times\Z^d\to\R$ the \emph{cocycle}  to which $\theta$ is the gradient: for $x\in\Z^d$ and $k\in\cE$
\begin{align}
\label{grad2}
\Theta(\omega, 0)=0, 
\hskip1cm
\Theta(\omega,x+k)-\Theta(\omega, x)=\theta_k(\tau_x\omega), 
\quad
\pi\text{-a.s.}.
\end{align}
Equations \eqref{harm1} and \eqref{grad2} amount to the fact that for all $x\in\Z^d$
\begin{align}
\notag
\phi(\tau_x\omega) - \sum_{k\in\cE}p_k(\tau_x\omega)\left(\Theta(\omega,x+k) - \Theta(\omega,x)\right)=0, 
\quad
\pi\text{-a.s.}
\end{align}
Hence, it follows that for $\pi$-a.a. $\omega\in\Omega$ fixed, the process 
\begin{align}
\label{martingale}
t\mapsto Y(t):= \int_0^t \phi(\tau_{X(s)}\omega)ds-\Theta(\omega, X(t))
\end{align}
is a martingale in the quenched filtration
\begin{align}
\notag
\cF_t:=\cF \vee \sigma\{X(s): 0\le s <t\}. 
\end{align}
That is: with $\omega\in\Omega$ fixed.

Due to the martingale central limit theorem and stationarity and ergodicity of the environment process $t\mapsto \eta(t)$ defined in \eqref{env proc} (see section 1.2 of \cite{kozma-toth_17}), the $\pi$-a.s. (quenched) central limit theorem follows for the process $t\mapsto Y(t)$. 

\begin{proposition}
\label{prop:martingale clt}
Conditions \eqref{bistoch}, \eqref{ellipt}, \eqref{hcond} are assumed. 
Let $\phi\in\cH_{-1}$.
For $\pi$-a.a. $\omega\in\Omega$, and any bounded and continuous function $f:\R\to\R$,
\begin{align}
\notag
\lim_{t\to\infty}
\expectom{f(t^{-1/2}Y(t))}
=
\frac{1}{\sqrt{2\pi} \bar\sigma}
\int_{-\infty}^\infty e^{-y^2/(2\bar\sigma^2)} f(y) dy,
\end{align}
with variance 
\begin{align}
\label{covariance}
\bar\sigma^2
:=
\sum_{k\in\cE}
\int_\Omega  s_k(\omega) \theta_k(\omega)^2d\pi(\omega)
>0.
\end{align}
\end{proposition}

As a corollary of Proposition \ref{prop:martingale clt} we get the quenched CLT for the \emph{harmonic coordinates} (that is, the appropriately corrected displacement) of the random walker. Indeed, first note that due to (the first line in)\eqref{symmetries} $\psi\in(\cH_{-1})^d$ holds a priori, and due to the $H_{-1}$-condition \eqref{hcond} $\varphi\in(\cH_{-1})^d$. Actually this latter fact is one of the equivalent forms of the $H_{-1}$-condition, see (KT35). Hence the term.  Therefore we can choose
\begin{align}
\notag
\phi
=
\phi^*
:=
\varphi+\psi\in (\cH_{-1})^d.
\end{align}
and solve (coordinate-wise) equation \eqref{harm1} with $\phi^*_i$, $i=1,\dots, d$, on the right hand side. We denote the solution $\theta^*\in\cG^d$ and  define the $\R^d$-valued cocycle $\Theta^*$ by \eqref{grad2}, with $\theta^*$ as gradient. Now, let 
\begin{align}
\notag
Y^*(t):=X(t) - \Theta^*(\omega, X(t)). 
\end{align}

\begin{corollary}
\label{cor:martingale clt}
Conditions \eqref{bistoch}, \eqref{ellipt}, \eqref{hcond} are assumed. 
For $\pi$-a.a. $\omega\in\Omega$, and any bounded and continuous function $f:\R^d\to\R$,
\begin{align}
\notag
\lim_{t\to\infty}
\expectom{f(t^{-1/2}Y^*(t))}
=
(2\pi \det \bar\sigma^2)^{-d/2}
\int_{\R^d} e^{-\frac12 y\cdot \bar\sigma^{-2} y} f(y) dy,
\end{align}
with nondegenerate covariance matrix 
\begin{align}
\label{displ covariance}
\bar\sigma^2_{ij}
:=
\sum_{k\in\cE}
\int_\Omega  s_k(\omega) (\theta_k(\omega)-k)_i(\theta_k(\omega)-k)_j d\pi(\omega).
\end{align}
\end{corollary}

The quenched CLT with the correcting terms $\Theta(X(t))$ removed will follow from Proposition \ref{prop:martingale clt}/Corollary \ref{cor:martingale clt} and the following error estimate.

\begin{proposition}
\label{prop:small error}
Conditions \eqref{bistoch}, \eqref{ellipt}, \eqref{hcond}, \eqref{two-plus-eps} are assumed. 
Let $\Omega\times\Z^d\ni x\mapsto \Psi(\omega,x)\in\R$ be a square integrable zero-mean cocycle. For $\pi$-a.a. $\omega\in\Omega$ and any $\delta>0$,
\begin{align}
\label{key-error}
\lim_{t\to\infty}\probabom{\abs{\Psi( X(t))}>\delta \sqrt{t}}=0.
\end{align}
\end{proposition}

Indeed, Propositions \ref{prop:martingale clt}/Corollary \ref{cor:martingale clt} and Proposition \ref{prop:small error} readily imply the main result of this paper.

\begin{theorem}
\label{thm: quenched clt}
Conditions \eqref{bistoch}, \eqref{ellipt}, \eqref{hcond}, \eqref{two-plus-eps} are assumed. 
For $\pi$-a.a. $\omega\in\Omega$ the following quenched CLTs hold.
\\
(i)
Let $\phi\in\cH_{-1}$.
For any bounded and continuous function $f:\R\to\R$,
\begin{align}
\notag
\lim_{t\to\infty}
\expectom{f(t^{-1/2}\int_0^t\phi(\eta(s)) ds)}
=
\frac{1}{2\pi \bar\sigma}
\int_{-\infty}^\infty e^{-y^2/(2\bar\sigma^2)} f(y) dy,
\end{align}
with the variance $\bar\sigma^2$ given in \eqref{covariance}. 
\\
(ii)
For any bounded continuous function $f:\R^d\to\R$, 
\begin{align}
\notag
\lim_{t\to\infty}
\expectom{f(t^{-1/2}X(t))}
=
(2\pi \det \bar\sigma^2)^{-d/2} ()^{-1}\int_{\R^d} e^{-\frac{y\cdot \bar\sigma^{-2} y}{2}} f(y) dy,
\end{align}
with the non-degenerate covariance matrix $\bar\sigma^2$ given in \eqref{displ covariance}.

\end{theorem}

\medskip
\noindent
{\bf Remarks:} 
\begin{enumerate}[$\circ$]

\item
Theorem \ref{thm: quenched clt} readily extends to all finite dimensional marginals of the diffusively scaled \emph{process} $t\mapsto T^{-1/2} X(Tt)$, as $T\to\infty$. In order to spare notation and space we don't make explicit this straightforward extension. 

\item
The idea of \emph{harmonic coordinates} originates in the seminal paper \cite{kozlov_85}. However, as pointed out in later works (see e.g  \cite{sidoravicius-sznitman_04}, \cite{biskup_11} or \cite{komorowski-landim-olla_12}) beside the highly innovative ideas some arguments of key importance are not fully complete there. 

\item
By restricting to $p_k(\omega)=p_{-k}(\tau_k\omega)$ (that is, $p_k=s_k$, $v_k\equiv0$), see \eqref{symmetries}, the case of random walks among bounded and elliptic random conductances is covered. This is Theorem 1 in \cite{sidoravicius-sznitman_04}. However, since the ellipticity condition \eqref{ellipt} is essential in our current setup,  Theorem 2.1 of \cite{sidoravicius-sznitman_04} and the main results of \cite{berger-biskup_07}, \cite{biskup-prescott_07}, \cite{anders-deuschel-slowik_15} are not covered as special cases. 

\item
Relaxing the ellipticity condition \eqref{ellipt} within this context remains open. This might be possibly resolved by combining ideas and techniques from \cite{anders-deuschel-slowik_15}, \cite{biskup-prescott_07} with those in this paper.

\end{enumerate}

\bigskip
\noindent
Before turning to the proofs we summarize what is truly new -- compared with earlier works on quenched CLT for RWRE -- in the details that follow. 

\begin{enumerate}[$\circ$]

\item
The proof of Proposition \ref{prop:tightness} relies on an extension of Nash's celebrated moment bound, cf.\cite{nash_58}, to non-reversible, divergence-free drift (i.e. incompressible flow) context. To our knowledge this is the first such kind of extension of Nash's arguments. 

\item
In the proof of Proposition \ref{prop:harmonic_coordinates}, 
the construction of the harmonic coordinates is done by functional analytic tools, relying on the the method of \emph{relaxed sector condition}, cf. \cite{horvath-toth-veto_12}, \cite{kozma-toth_14}, \cite{kozma-toth_17}, which differs essentially from the methods employed in the cited earlier works. 

\item
In the proof of Proposition \ref{prop:small error}, softer than usual, merely ergodic  arguments are employed in proving vanishing under diffusive scaling of the corrector term. 

\end{enumerate}

\section{Proofs}

\subsection{Tightness:  Proof of Proposition \ref{prop:tightness}}

We follow Nash's blueprint, cf. \cite{nash_58}. See also \cite{bass_02} for a streamlined version of the proof and \cite{barlow_04}, \cite{biskup-prescott_07} for adaptation of details to lattice walk on $\Z^d$ (rather than continuous diffusion on $\R^d$) setup. However, new elements are needed due to the non-reversible drift term. These new elements of the proof will be highlighted. 

The main ideas of \cite{nash_58} have been employed in  the context of random walks among random conductancies, cf. \cite{barlow_04}, \cite{biskup-prescott_07}. In all cited works, however, the diffusions and random walks considered have been \emph{reversible} with respect to uniform measure on $\R^d$, respectively, $\Z^d$. That is, the diffusion generators in \cite{nash_58} and \cite{bass_02} are in divergence form, the random walks in \cite{barlow_04} and \cite{biskup-prescott_07} are defined by conductancies of unoriented edges. It has been well known that the diagonal heat kernel upper bound, \eqref{nassh} below, follows  from Nash's inequality not only in the reversible but also in the doubly stochastic/divergence-free cases. The novelty in Proposition \ref{prop:tightness} and its forthcoming proof is the extension of the entropy and entropy-production bounds of \cite{nash_58} to the nonreversible case, with doubly stochastic (or, sourceless, divergence-free, incompressible) jump rates. In the diffusion setup this corresponds to a \emph{divergence-free} drift term added to the reversible infinitesimal generator. The main point is, that in this case and under the integrability condition \eqref{two-plus-eps} we are able to control the terms coming from the skew-self-adjoint parts, too, by the entropy production. This is by no means straightforward. Without assuming at least the $H_{-1}$-condition \eqref{hcond} this moment bound is simply not valid, even in the annealed setup. The stronger integrability condition imposed on the stream tensor may well be a technical nuisance only. 

All constants in the forthcoming estimates will depend on $d$, $s_*$, $s^*$, $\varepsilon$ and $h^*$ only. See \eqref{ellipt} and \eqref{hbound} for their definition. We will adopt the following notational convention: those constants where their being positive (but possibly small) is important will be denoted by lower case symbols $c_j$, whereas those ones where their being finite (but possibly large) is the point will be denoted by upper case symbols $C_j$. We will be explicit about which constants depend on which of the four parameters listed above. 

We are in the quenched setup. However, in order to lighten notation dependence on $\omega\in\Omega$  will not be shown explicitly within this proof. Denote 
\begin{align}
\notag
q(t,x)
=
q(t,x,\omega)
&:=
\probabom{X(t)=x}, 
\\
\notag
M(t)
=
M(t,\omega)
&:=
\expectom{\abs{X(t)}}
=
\sum_{x\in\Z^d} \abs{x}q(t,x)
\\
\notag
H(t)
=
H(t,\omega)
&:=
-\sum_{x\in\Z^d} \log q(t,x) q(t,x).
\end{align}

The ingredients of the proof of Proposition \ref{prop:tightness} are collected in lemmas \ref{lem:entropyineq}, \ref{lem:nashbound} and \ref{lem:derivative bound} below. For the proofs of lemmas \ref{lem:entropyineq} and \ref{lem:nashbound} we refer to earlier works. We present full proof of Lemma \ref{lem:derivative bound} which contains the new elements. 

\begin{lemma}
\label{lem:entropyineq}
There exists a constant $c_1=c_1(d)\in(0,\infty)$  such that for any $t>0$ it holds that if $M(t)>1$ then 
\begin{align}
\label{entropy}
M(t)\ge c_1 e^{\frac{H(t)}{d}}.
\end{align}
\end{lemma}

\smallskip
\noindent
The bound \eqref{entropy} is direct consequence of the entropy inequality and it is actually valid for any probability distribution $q(x)$ on $\Z^d$. 
See \cite{nash_58}, \cite{bass_02} for a proof for absolutely continuous probability measures on $\R^d$ and \cite{barlow_04}, \cite{biskup-prescott_07} for its adaptation to probability measures on $\Z^d$. We do not reproduce here these details. It is interesting to note that in \cite{nash_58} Nash attributes this particular argument to Carleson.

\qed

\begin{lemma}
\label{lem:nashbound}
There exists a constant $C_2=C_2(d, s_*)\in(0,\infty)$ such that 
\begin{align}
\label{nassssh}
\frac{H(t)}{d} \ge \frac12\log t -C_2.
\end{align}
\end{lemma}

\smallskip
\noindent
From Nash's inequality it follows, that there exists a constant $C=C(d,s_*)$ such that for all $t\ge0$ and $x\in\Z^d$
\begin{align}
\label{nassh}
q(t,x)\le C t^{-d/2}.
\end{align}
See Proposition 3 in \cite{kozma-toth_17} for an alternative derivation using "evolving sets" method of \cite{morris-peres_05}. We omit the details. The bound \eqref{nassssh} follows directly from \eqref{nassh} and the definition of the entropy $H(t)$. 

\qed

Defining now
\begin{align}
\notag
G(t):=\frac{H(t)}{d} - \frac12\log t +C_2 \ge 0, 
\end{align}
the entropy bound \eqref{entropy} reads
\begin{align}
\label{entropy2}
t^{-1/2} M(t)\ge c_3 e^{G(t)},
\end{align}
with $c_3=c_3(d,s_*)=c_1 e^{-C_2}\in(0,\infty)$. 

\begin{lemma}
\label{lem:derivative bound}
There exists a constant $C_4=C_4(d,s_*, s^*, h^*)\in(0,\infty)$ so that for $\pi$-almost all $\omega\in\Omega$ there exists $t^*(\omega)<\infty$ such that for $t>t^*(\omega)$
\begin{align}
\label{derivative bound}
t^{-1/2} M(t) \le C_4 (G(t) + \vareps^{-1})^{\frac{1+\vareps}{2+\vareps}}, 
\end{align}
where $\vareps>0$ is from \eqref{two-plus-eps}. 
\end{lemma}

\smallskip
\noindent
{\bf Remark.}
Letting $\vareps\to0$, $h^*=h^*(\vareps)$  decreases to $\sum_{kl\in\cE}\norm{h_{k,l}}_2<\infty$. Therefore $C_4$ also decreases to a finite positive limit. Nonetheless, the right hand side of \eqref{derivative bound} blows up due to the $\vareps^{-1}$ term. This is the reason of imposing \eqref{two-plus-eps}, with $\vareps>0$.

\begin{proof}
[Proof of Lemma \ref{lem:derivative bound}]
Within this proof we will use the notation 
\begin{align}
\notag
s_k(x):=s_k(\tau_x\omega), 
\qquad
v_k(x):=v_k(\tau_x\omega), 
\qquad
h_{k,l}(x):=h_{k,l}(\tau_x\omega).
\end{align}

In the following computations we use repeatedly Kolmogorov's forward equation 
\begin{align}
\label{kfe}
\dot q(t,x)
=
\frac12
\sum_{x\in\Z^d, k\in\cE}
s_k(x) 
(q(t,x+k)-q(t,x))
+
\frac12
\sum_{x\in\Z^d, k\in\cE}
v_k(x) 
(q(t,x+k)+q(t,x)).
\end{align}
In the last term the divergence-freeness \eqref{divfree} of $v$ is used. 

First we provide a lower bound on $\dot H(t)$:
\begin{align}
\notag
\dot H(t)
=
&
\phantom{-cs^*.}
\frac12 
\sum_{x\in\Z^d, k\in\cE}
s_k(x) 
(q(t,x+k)-q(t,x)) (\log q(t,x+k)-\log q(t,x))
\\
\notag
&
\phantom{cs^*}
-
\frac12 
\sum_{x\in\Z^d, k\in\cE}
v_k(x)
(q(t,x+k)+q(t,x)) (\log q(t,x+k)-\log q(t,x))
\\
\notag
&
\phantom{cs^* \frac12}
+ 
\sum_{x\in\Z^d, k\in\cE}
v_k(x)
(q(t,x+k)-q(t,x))
\\
\notag
=
&
\phantom{-cs^*.}
\frac12 
\sum_{x\in\Z^d, k\in\cE}
s_k(x) (q(t,x+k)-q(t,x)) 
\int_{q(t,x)}^{q(t,x+k)} \frac{1}{u} du
\\
\notag
&
\phantom{cs^*}
-
\frac12 
\sum_{x\in\Z^d, k\in\cE}
v_k(x)
\int_{q(t,x)}^{q(t,x+k)} \frac{q_t(x)+q_t(x+k)-2u}{u} du.
\\
\notag
\ge
&
\phantom{-\frac{1}{2}c.}
s_*
\sum_{x\in\Z^d, k\in\cE}
\int_{q(t,x)\land q(t,x+k)}^{q(t,x)\lor q(t,x+k)} \frac{u-q(t,x)\land q(t,x+k)}{u} du
\\
\label{entropyprod}
\ge
&
\phantom{-\frac{1}{2}.}
c_5
\sum_{x\in\Z^d, k\in\cE}
\abs{\frac{q(t,x+k)-q(t,x)}{q(t,x+k)+q(t,x)}}^2q(t,x).
\end{align}
The first step follows from from \eqref{kfe} by explicit computations, using the symmetries \eqref{symmetries} of $s$ and $v$, and also the divergence-freeness of $v$, \eqref{divfree}. Note, that due to this latter the third sum on the right hand side vanishes. We included it as a dummy.  
The second step is just transcription of differences to appropriate integrals. 
In the third step we have used $s_k\ge s_*\lor \abs{v_k}$. 
Finally, in the last step we have used the bound
\begin{align}
\notag
b
:=
\inf_{1<\beta<\infty}
\frac{\beta+1}{(\beta-1)^2}\int_1^\beta \frac{u-1}{u} du 
=
0.8956\dots
> 0,
\end{align}
and got $c_5=c_5(s_*)=b s_*$.
Note, that the lower bound on entropy production in terms of Fisher-entropy, \eqref{entropyprod}, looks formally the same as in the reversible case. However, deriving it, one has to control the skew-symmetric part by the symmetric part of the entropy production. This can be done due to incompressibility (or sourcelessness, or divergence-freeness) of the flow $v$. 

Next we compute $\dot M(t)$.
\begin{align}
\notag
\dot M(t)
=
&
\phantom{-.}
\frac12 
\sum_{x\in\Z^d, k\in\cE}
s_k(x) 
(\abs{x}-\abs{x+k})(q(t,x+k)-q(t,x)) 
\\
\notag
&
-
\frac12 
\sum_{x\in\Z^d, k\in\cE}
v_k(x)
(\abs{x}+\abs{x+k})(q(t,x+k)-q(t,x)) 
\\
\notag
=
&
\phantom{-.}
\frac12 
\sum_{x\in\Z^d, k\in\cE}
s_k(x) 
(\abs{x}-\abs{x+k})(q(t,x+k)-q(t,x)) 
\\
\notag
&
-
\frac12 
\sum_{x\in\Z^d, k,l\in\cE}
h_{k,l}(x)
(\abs{x+k}-\abs{x+l})(q(t,x+k+l)-q(t,x)) 
\end{align}
The first step follows from \eqref{kfe} by explicit computation, using the symmetries \eqref{symmetries} of $s$ and $v$. The second step follows from \eqref{vcurl} and the symmetries \eqref{htensor}.
Hence, 
\begin{align}
\notag
\abs{\dot M(t)}
\le
&
C_6
\sum_{x\in\Z^d, k\in\cE}
\left(s^* + \sum_{l\in\cE}\abs{h_{k,l}(x)}\right)
\abs{\frac{q(t,x+k)-q(t,x)}{q(t,x+k)+q(t,x)}}q(t,x),
\end{align}
with $C_6=C_6(d)$.

Integrating over $t$ and applying Minkowski's inequality we obtain
\begin{align}
\notag
\abs{M(t)}
\le
&
C_6
t^{\frac{1}{2+\vareps}}
\left(\frac{1}{t} \int_0^t \sum_{x\in\Z^d, k \in\cE} \left(s^*+\sum_{l \in\cE}\abs{h_{k,l}(x)}\right)^{2+\vareps} q(u,x) du\right)^{\frac{1}{2+\vareps}}
\\
\label{absmt}
&
\phantom{t^{\frac{1}{2+\vareps}}}
\times
\left(\int_0^t \sum_{x\in\Z^d, k \in\cE}\abs{\frac{q(u,x+k)-q(u,x)}{q(u,x+k)+q(u,x)}}^{\frac{2+\vareps}{1+\vareps}} q(u,x) du\right)^{\frac{1+\vareps}{2+\vareps}}
\end{align}

Due to the (Hopf-) Chacon-Ornstein ergodic theorem (see \cite{hopf_54}, \cite{chacon-ornstein_60}, \cite{hopf_60} or \cite{krengel_85}) and integrability of $\abs{h_{k,l}}^{2+\vareps}$ the middle factor in \eqref{absmt} converges to a finite deterministic value, for $\pi$-almost all $\omega\in\Omega$, as $t\to\infty$. Indeed, 
\begin{align}
\notag
&
\frac{1}{t} \int_0^t 
\sum_{x\in\Z^d, k \in\cE} 
\left(s^*+\sum_{l \in\cE}\abs{h_{k,l}(x)}\right)^{2+\vareps} 
q(u,x) 
du
=
\\
\label{co}
&\hskip4cm
\frac{1}{t} \int_0^t 
\sum_{k \in\cE} 
\expectom{\left(s^*+\sum_{l \in\cE}\abs{h_{k,l}(\eta(u))}\right)^{2+\vareps}} 
du,  
\end{align}
where $t\mapsto \eta(t)$ is the Markov process of the environment seen by the random walker, defined in \eqref{env proc}, which is stationary and ergodic on $(\Omega, \pi)$. This is the typical context for the (Hopf-) Chacon-Ornstein theorem. The right hand side of \eqref{co} $\pi$-almost-surely converges to 
\begin{align}
\notag
C_7^{2+\vareps}
:=
\sum_{k \in\cE} 
\int_{\Omega} 
\left (s^*+ \sum_{l \in\cE} \abs{h_{k,l}(\omega)}\right)^{2+\vareps} 
d\pi(\omega)
<\infty.
\end{align}
Obviously, $C_7=C_7(d, s^*, h^*)$.
Note that this is the only argument where the stronger integrability condition \eqref{two-plus-eps} is used. 

On the other hand, due to \eqref{entropyprod} and a H\"older bound, the last factor in \eqref{absmt} is dominated by the entropy production. Altogether we obtain that for $\pi$-almost all $\omega\in\Omega$, there exists $t^*(\omega)<\infty$, such that for all $t>t^*(\omega)$.
\begin{align}
\label{altogether}
\abs{M(t)}
\le
&
C_8
t^{\frac{1}{2+\vareps}}
\left(\int_0^t \dot H(u)^{\frac{2+\vareps}{2+2\vareps}} du\right)^{\frac{1+\vareps}{2+\vareps}},
\end{align}
where $C_8=C_8(d, s_*, s^*, h^*):= 2 C_6 C_7 c_5^{-1/2}$.

The rest is straight sailing. Following \cite{nash_58}, with due modifications,  
\begin{align}
\notag
\int_0^t \dot H(u)^{\frac{2+\vareps}{2+2\vareps}} du
&=
\int_0^t \left(\dot G(u) + \frac{1}{2u}\right)^{\frac{2+\vareps}{2+2\vareps}} du
\\
\notag
&=
\int_0^t (2u)^{-\frac{2+\vareps}{2+2\vareps}} \left( 1 +  2u \dot G(u) \right)^{\frac{2+\vareps}{2+2\vareps}} du
\\
\notag
&\le
\int_0^t \left( (2u)^{-\frac{2+\vareps}{2+2\vareps}} + \frac{2+\vareps}{2+2\vareps}  (2u)^\frac{\vareps}{2+2\vareps} \dot G(u) \right) du
\\
\notag
&=
\frac{2+2\vareps}{\vareps} t^{\frac{\vareps} {2+2\vareps}} + \frac{2+\vareps}{2+2\vareps} t^{\frac{\vareps} {2+2\vareps}} G(t) 
-
\frac{\vareps}{2+2\vareps} \int_0^t (2u)^{-\frac{2+\vareps}{2+2\vareps}} G(u) ds
\\
\notag
&\le
3t^{\frac{\vareps} {2+2\vareps}} \left(\vareps^{-1} + G(t)\right).
\end{align}
Inserting this into \eqref{altogether} we obtain \eqref{derivative bound} of Lemma \ref{lem:derivative bound}, with $C_4=3 C_8$. 
\end{proof}

To conclude the proof of Proposition \ref{prop:tightness} note that \eqref{entropy2} and \eqref{derivative bound} jointly imply that there exists a constant $C_9=C_9(\varepsilon, d, s_*, s^*, h^*)$ so that for $\pi$-almost all $\omega\in\Omega$, there exists a $t^*(\omega)$ so that for $t>t^*(\omega)$, $G(t)\le C_9$. Hence follows \eqref{tightness}, via \eqref{derivative bound}.

\qed

\subsection{Some operators over $\cH$ and $\cG$}

First we recall from \cite{kozma-toth_17} some bounded operators acting on the Hilbert spaces $\cH$ and $\cG$. 

Let $(\Omega,\cF, \pi, \tau_z:z\in\Z^d)$ be a probability space with an is ergodic $\Z^d$-action. The gradient, Laplacian and Riesz operators are all directly expressed with the help of the shift operators $U_kf(\omega):= f(\tau_k\omega)$, as follows.

\smallskip
\noindent
$\nabla_k, \Delta, \Gamma_k: \cH\to\cH$: 
\begin{align}
\label{grad-lap-riesz}
\nabla_k:=U_k-I,
&&
\Delta:=2\sum_{k\in\cE}\nabla_k,
&&
\Gamma_k:=\abs{\Delta}^{-1/2}\nabla_k
\end{align}

\smallskip
\noindent
$\nabla, \Gamma: \cH\to\cG$:
\begin{align}
\notag
(\nabla f)_k:=\nabla_k f, 
&&
(\Gamma f)_k:=\Gamma_k f.
\end{align}

\smallskip
\noindent
$\nabla^*, \Gamma^*: \cG\to\cH$:
\begin{align}
\notag
\nabla^* g:=\sum_{k\in\cE}\nabla_{-k}g_k, 
&&
\Gamma^* g:=\sum_{k\in\cE}\Gamma_{-k}g_k, 
\end{align}
The following identities hold, 
\begin{align}
\label{isis}
\nabla^*\nabla = - \Delta
&&
\Gamma^*\Gamma = I_{\cH}, 
&&
\Gamma\Gamma^* = I_{\cG}. 
\end{align}
The first two follow directly from the definitions and straightforward computations. The proof of the third one relies on $\Ker(\nabla^*)=\Ker(\Gamma^*)=\{0_{\cG}\}$. This follows from ergodicity, $\Ker(\Delta)=\{0_{\cH}\}$, and its proof is left as an exercise. The last two identities in \eqref{isis} mean that $\Gamma:\cH\to\cG$ is an \emph{isometric isomorphism}. This fact will have importance below. 

We will also use the multiplication operators $M_k, N_k: \cL^2(\Omega, \pi)\to\cL^2(\Omega, \pi)$, $k\in\cE$ (see (KT38), (KT39)): 
\begin{align}
\label{multiplop}
&
N_kf(\omega):=\left(s_k(\omega)-s_*\right)f(\omega), 
&&
M_kf(\omega):=v_k(\omega)f(\omega), 
\end{align}
and recall the commutation relations (KT40):
\begin{align}
\label{commute}
\begin{aligned}
&
-\sum_{k\in\cE}N_k\nabla_k 
=
-\sum_{k\in\cE}\nabla_{-k}N_k
=
\frac12
\sum_{k\in\cE}\nabla_{-k}N_k\nabla_{k}
=: T=T^*\ge0,
&
\\
&
\phantom{-\,}
\sum_{k\in\cE}M_k\nabla_k
=
- 
\sum_{k\in\cE}\nabla_{-k}M_k, 
=:A=-A^*,
&
\end{aligned}
\end{align}
which follow directly from \eqref{symmetries} and \eqref{divfree}. 

Strictly speaking, the multiplication operators $M_k$ and $N_k$ do not preserve the subspace $\cH\subset\cL^2(\Omega,\pi)$ of zero mean elements. However, they only appear in the combinations $\sum_{k\in\cE}N_k\nabla_k$, respectively, $\sum_{k\in\cE}M_k\nabla_k$, which due to the commutation relations \eqref{commute} do preserve $\cH$.

Also recall the decomposition of the infinitesimal generator $L$ of the environment process $t\mapsto\eta(t)$ into self-adjoint and skew-self-adjoint parts (cf. (KT41)-(KT43)):
\begin{align}
\notag
L
=
\frac12\Delta-T+A
=
-
S+A
.
\end{align}
Note that the (absolute value) of the Laplacian minorizes and majorizes the self-adjoint part of the infinitesimal generator: 
\begin{align}
\label{domin}
s_* \abs{\Delta} \le 2 S \le s^* \abs{\Delta}.
\end{align}
The inequalities are meant in operator sense. The ellipticity condition \eqref{ellipt} is used in the lower bound, and bounded jump rates in the upper bound. 

\subsection{Harmonic coordinates: Proof of Proposition \ref{prop:harmonic_coordinates}}

Since $\Gamma:\cH\to\cG$ is an isometric isomorphism (see \eqref{isis}) we can assume that
\begin{align}
\notag
\theta=\Gamma \chi, 
\end{align}
with some $\chi\in\cH$, and write the equation \eqref{harm1} for $\chi\in\cH$ as follows:
\begin{align}
\label{eqforchi}
\left(\abs{\Delta}^{1/2} + \sum_{k\in\cE} N_k\Gamma_k + \sum_{k\in\cE} M_k\Gamma_k \right) \chi
=
\phi.
\end{align}
This is the equation to be solved for $\chi\in\cH$. 

In order to present the argument in its most transparent form let's first make the simplifying assumption that the symmetric part $s_k$ of the jump rates $p_k$ (see (KT5)) are actually constant, $s_k(\omega)\equiv 1$ $\pi$-a.s.:
\begin{align}
\label{constsymm}
p_k(\omega)= 1 + v_k(\omega). 
\end{align}
This is the case treated in an early arxive version of \cite{kozma-toth_17} available at \cite{kozma-toth_14}. Its advantage is that the relevant ideas appear in their most transparent form, without the formal (but unessential) complications caused by the non-constant symmetric parts. In this case we have (see \eqref{multiplop})
\begin{align}
\notag
N_k=0, \  \text{for all }k\in\cE.
\end{align}
Thus \eqref{eqforchi} reduces to 
\begin{align}
\label{eqforchi_simplified}
\left(\abs{\Delta}^{1/2}  + \sum_{k\in\cE} M_k\Gamma_k\right) \chi
=
\phi.
\end{align}
Since it is assumed that $\phi\in\cH_{-1}$, we can multiply equation \eqref{eqforchi_simplified} from the left by $\abs{\Delta}^{-1/2}$ to get 
\begin{align}
\label{eqforchi_2}
\left(I +  \abs{\Delta}^{-1/2}\sum_{k\in\cE} M_k\Gamma_k \right) \chi
=
\abs{\Delta}^{-1/2}\phi.
\end{align}
On the left hand side of this equation we have exactly the densely defined and closed \emph{unbounded} operator 
\begin{align}
\notag
-B^*
:=
\abs{\Delta}^{-1/2}\sum_{k\in\cE} M_k\Gamma_k
\end{align}
(see (KT58)) which in  Proposition 2 of \cite{kozma-toth_14} is proved to be \emph{skew-self-adjoint} (not merely the adjoint of a skew-symmetric one). Recall that this is the key technical point in the proof of the main result in \cite{kozma-toth_14}. Thus, the spectrum of the operator $B=-B^*$ is on the imaginary axis, and therefore on the left hand side of \eqref{eqforchi_2} $I-B^*=I+B$ is \emph{invertible}, the unique solution of \eqref{eqforchi_simplified} being 
\begin{align}
\notag
\chi = (I+B)^{-1} \left(\abs{\Delta}^{-1/2} \phi\right).
\end{align}
Finally 
\begin{align}
\notag
\theta_k = \Gamma_k \left(I+B\right)^{-1} \left(\abs{\Delta}^{-1/2} \phi\right), 
\ \ \ 
k\in\cE.
\end{align}
These are bona fide elements of $\cH$, since 
\begin{align}
\notag
\abs{\Delta}^{-1/2} \phi \in\cH, 
\qquad
\norm{(I+B)^{-1}} \le 1, 
\qquad
\norm{\Gamma_k}\le1. 
\end{align}

Now we go to the general case, without assuming \eqref{constsymm}. It is proved in Theorem RSC2 of \cite{kozma-toth_17} that due to \eqref{domin} the operator $\abs{\Delta}^{1/2} S^{-1/2}$ is bounded and has a bounded inverse, and the a priori densely defined operator $C:=S^{-1/2} A S^{-1/2}$ is \emph{essentially skew-self-adjoint}. (See the proof of Theorem RSC2 in the Appendix of \cite{kozma-toth_17}.) Recall that this is the key technical point in the proof of the main result of \cite{kozma-toth_17}. Hence it follows that 
\begin{align}
\label{sln}
\chi:=
\left( \abs{\Delta}^{1/2} S^{-1/2} \right) 
\left( I+ C \right)^{-1}
\left( S^{-1/2} \abs{\Delta}^{1/2} \right) 
\abs{\Delta}^{-1/2}\phi
\end{align}
is a bona fide element of $\cH$. Indeed, 
\begin{align}
\notag
\norm{S^{-1/2} \abs{\Delta}^{1/2}}
\!
=
\!
\norm{\abs{\Delta}^{1/2} S^{-1/2}}
<\infty, 
\quad
\norm{\left( I+ C \right)^{-1}}
\le1, 
\quad
\abs{\Delta}^{-1/2}\phi \in\cH. 
\end{align}

It is an easy formal computation to check that $\chi$ in \eqref{sln} provides the solution to the equation \eqref{eqforchi} in the general case, and hence 
\begin{align}
\notag
\theta_k 
&
= 
\Gamma_k
\left( \abs{\Delta}^{1/2} S^{-1/2} \right) 
\left( I+ C \right)^{-1}
\left( S^{-1/2} \abs{\Delta}^{1/2} \right) 
\abs{\Delta}^{-1/2}\phi.
\end{align}

\qed

\subsection{Martingale CLT: Proof of Proposition \ref{prop:martingale clt} and Corollary \ref{cor:martingale clt}}

This follows from the most conventional application of the martingale central limit theorem, see e.g. \cite{hall-heyde_80}, \cite{helland_82}. Due to the choice of $\theta$, for $\pi$-a.a. $\omega\in\Omega$ the quenched process $t\mapsto Y(t)$ defined in \eqref{martingale} is a martingale. Its infinitesimal conditional variance process is 
\begin{align}
\notag
\lim_{h\to 0}h^{-1}\condexpectom{(Y(t+h)-Y(t))^2}{\cF_t}
=
\sigma^2(\eta(t))
\end{align}
where $t\mapsto\eta(t):=\tau_{X(t)}\omega$ is the environment process as seen by the random walk, defined in \eqref{env proc}, and $\sigma^2:\Omega\to \R_+$ is
\begin{align}
\notag
\sigma^2(\omega)
=
\sum_{k\in\cE}
p_k(\omega)\abs{\theta_k(\omega)}^2.
\end{align}
The key observation is that since the Markov process $t\mapsto\eta(t)$ is stationary and ergodic (see section 1.2 of \cite{kozma-toth_17}) the following strong law of large numbers holds:
\begin{align}
\notag
\lim_{t\to\infty}\frac{1}{t}\int_0^t\sigma^2(\eta(s))ds
=
\int_\Omega \sigma^2(\omega) d\pi(\omega)
=:\bar\sigma^2, 
\qquad
\pi\text{-a.s.}
\end{align}
Positivity  of the variance $\bar\sigma^2$ follows from the the (skew)symmetry of $v$ in the second line of \eqref{symmetries} and the ellipticity condition \eqref{ellipt}. Indeed, from these relations it follows that
\begin{align}
\notag
\bar\sigma^2
=
\sum_{k\in\cE}
\int_\Omega
p_k(\omega) \abs{\theta_k(\omega)}^2
d\pi(\omega)
=
\sum_{k\in\cE}
\int_\Omega
s_k(\omega) \abs{\theta_k(\omega)}^2
d\pi(\omega)
\ge 
s_* 
\sum_{k\in\cE}
\int_\Omega
\abs{\theta_k(\omega)}^2 
d\pi(\omega)
>0. 
\end{align}
In the middle equality the symmetries \eqref{symmetries} are used. This concludes the proof of Proposition \ref{prop:martingale clt}. 

Corollary \ref{cor:martingale clt} follows directly. We apply the standard martingale decomposition (see (KT25)) \emph{and} Proposition \ref{prop:martingale clt}:
\begin{align}
\notag
Y^*(t)
=
\left(
X(t) - \int_0^t\phi^*(\tau_{X(s)}\omega)\, ds)
\right)
+
\left(
\int_0^t\phi^*(\tau_{X(s)}\omega)\, ds
-
\Theta^*(\omega, X(t))
\right),
\end{align}
and note that the martingale CLT applies. The expression \eqref{displ covariance} of the asymptotic covariance matrix follows as above. 

\qed

\subsection{Asymptotically vanishing corrector: Proof of Proposition \ref{prop:small error}}

We write (like in (KT74))
\begin{align}
\notag
\probabom{ \abs{\Psi(X(t))}>\delta \sqrt{t}}
&
\le
\probabom{ \{\abs{\Psi(X(t))}>\delta\sqrt{t}\} \cap \{\abs{X(t)}\le K\sqrt{t}\}}
+
\probabom{ \abs{X(t)}> K\sqrt{t} }
\\
\notag
&
\le
\delta^{-1}t^{-1/2}
\expectom{\abs{\Psi(X(t))} \ind{\abs{X(t)}\le K\sqrt{t}}}
+
K^{-1}t^{-1/2}
\expectom{\abs{X(t)}}.
\end{align}
Using the diagonal heat kernel upper bound \eqref{nassh} in the first term and the moment bound \eqref{tightness} in the second term from here we readily obtain 
\begin{align}
\label{tbc}
\varlimsup_{t\to\infty}
\probabom{ \abs{\Psi(X(t))}>\delta \sqrt{t}}
\le
C \delta^{-1} \varlimsup_{t\to\infty}
t^{-(1+d)/2} \sum_{\abs{x}\le K \sqrt{t}} \abs{\Psi(x)}
+
M^*K^{-1}.
\end{align}
The statement of Proposition \ref{prop:small error}, \eqref{key-error}, will follow from the following strong law of large numbers: 

\begin{lemma}
\label{lem:lln}
Let $(\Omega, \cF, \pi, \tau_z:z\in\Z^d)$ be a probability space with an ergodic $\Z^d$-action and $\Omega\times\Z^d\ni x\mapsto \Psi(\omega,x)\in\R$ be a zero-mean $\cL^2$-cocycle. Then
\begin{align}
\label{erglln}
\lim_{N\to\infty} N^{-(d+1)} \sum_{\abs{x}\le N} \abs{\Psi(x)}
=0, 
\qquad
\pi\text{-a.s.}
\end{align}
\end{lemma}

\noindent
{\bf Remarks on Lemma \ref{lem:lln}:}
\begin{enumerate}[$\circ$]

\item
The statement Lemma \ref{lem:lln} holds true actually for zero-mean $\cL^p$-cocycles, with $p>1$. However, here we only need the $\cL^2$ version. 

\item
The weaker statement 
\begin{align}
\label{werglln}
\lim_{N\to\infty} N^{-d} \sum_{\abs{x}\le N} \ind{\abs{\Psi(x)}>\vareps N}
=0, 
\qquad
\pi\text{-a.s.}, \quad \forall \vareps>0, 
\end{align}
readily follows from \eqref{erglln} by Markov's inequality. 

\item
Various versions of \eqref{erglln} or \eqref{werglln} appear as key ingredient in all proofs of quenched CLT for random walks among random conductances. As examples (in chronological order) see (0.13) (1.23) in \cite{sidoravicius-sznitman_04}; (5.15) in \cite{berger-biskup_07}; (2.15) and (5.25) in \cite{biskup-prescott_07}; (7.17) in \cite{kumagai_14}; (12) in \cite{anders-deuschel-slowik_15}; (4.1) in \cite{biskup-rodriguez_17}. (The list is certainly not exhaustive.) However, it seems to be the case that in all these works heavier tools had been used than the merely ergodic arguments employed in the proof below. This is our reason to include it here. 

\end{enumerate}

\begin{proof}
[Proof of Lemma \ref{lem:lln}.]
We will prove the lemma by induction on the dimension $d$ and for the sequence of cubic boxes $[0,N-1]^d$ rather than $[-N,N]^d$. For $d=1$ the statement of the Lemma is a direct consequence of Birkhoff's ergodic theorem. We will use the notation $(\un,m)\in[0,N-1]^{d}\times [0,N-1]$. Fix $L<\infty$  and write
\begin{align}
\label{threeterms}
&
\sum_{\un\in[0,N-1]^{d}} \sum_{m\in[0,N-1]} \abs{\Psi(\un,m)}
\le
\sum_{\un\in[0,N-1]^{d}} \sum_{l=0}^{L-1} \sum_{j=0}^{\lfloor(N-1)/L\rfloor} \abs{\Psi(\un, l+jL)}
\\
\notag
&\hskip2cm
\le
N \sum_{\un\in[0,N-1]^{d}}\abs{\Psi(\un, 0)}
+ 
\frac{N}{L} \sum_{\un\in[0,N-1]^{d}} \sum_{l=0}^{L-1} \abs{\Psi(\un, l)-\Psi(\un, 0)} 
\\
\notag
&\hskip3cm
+
\sum_{\un\in[0,N-1]^{d}} \sum_{l=0}^{L-1} \sum_{j=1}^{\lfloor(N-1)/L\rfloor} \sum_{i=0}^{j-1}\abs{\Psi(\un, l+(i+1)L)- \Psi(\un, l + iL)}.
\end{align} 
By the induction hypothesis, for the first term we get: 
\begin{align}
\notag
\lim_{N\to\infty}
N^{-(d+2)}N \sum_{\un\in[0,N-1]^{d}}\abs{\Psi(\un, 0)}
=
\lim_{N\to\infty}
N^{-(d+1)} \sum_{\un\in[0,N-1]^{d}}\abs{\Psi(\un, 0)} =0.
\end{align}
For the second term we apply directly the multidimensional version of the almost sure ergodic theorem: 
\begin{align}
\notag
\lim_{N\to\infty}
&
N^{-(d+2)}\frac{N}{L} \sum_{\un\in[0,N-1]^{d}} \sum_{l=0}^{L-1} \abs{\Psi(\un, l)-\Psi(\un, 0)} 
\\
\notag
&
\hskip3cm
=
L^{-1}\sum_{l=0}^{L-1}  
\lim_{N\to\infty}
N^{-d-1}
\sum_{\un\in[0,N-1]^{d}} 
\abs{\Psi(\un, l)-\Psi(\un, 0)} 
=0.
\end{align}
Finally, we turn to the third term on the right hand side of \eqref{threeterms}. 
\begin{align}
\notag
\varlimsup_{N\to\infty}
&
N^{-(d+2)}\sum_{\un\in[0,N-1]^{d}} \sum_{l=0}^{L-1} \sum_{j=1}^{\lfloor(N-1)/L\rfloor} \sum_{i=0}^{j-1}\abs{\Psi(\un, l+(i+1)L)- \Psi(\un, l + iL)}
\\
\notag
&
\le
\frac{1}{L}\sum_{l=0}^{L-1} 
\lim_{N\to\infty}
\frac{L^2}{N^2}\sum_{j=1}^{\lfloor(N-1)/L\rfloor} j
\frac{1}{N^{d}j} \sum_{\un\in[0,N-1]^{d}} \sum_{i=0}^{j-1} \frac{\abs{\Psi(\un, l+(i+1)L)- \Psi(\un, l + iL)}}{L}
\\
\notag
&
=
L^{-1}\expect{\abs{\Psi(\underline 0, L)- \Psi(\underline 0, 0)}}.
\end{align}
In the second step we have applied the multidimensional \emph{unrestricted} almost sure ergodic theorem, see Theorem 6.1.2 of \cite{krengel_85}. 

Finally, letting $L\to\infty$, by the multidimensional version of the mean ergodic theorem we obtain \eqref{erglln} in dimension $d+1$.
\end{proof}
Going now back to \eqref{tbc}, first applying \eqref{erglln} and then letting $K\to\infty$ we obtain \eqref{key-error}.

\qed

\section*{Acknowledgements}

Thanks are due to Marek Biskup and Takashi Kumagai for insisting on the question of extending the result of \cite{kozma-toth_17} to quenched setup and for their helpful remarks on the context of Lemma 1. I also thank Gady Kozma's illuminating comments. 
\\
This work was supported by EPSRC (UK) Fellowship EP/P003656/1, by The Leverhulme Trust (UK) through the International Network Laplacians, Random Walks,
Quantum Spin Systems and by OTKA (HU) K-109684.

\vskip2cm

\hbox{
\hskip9cm
\vbox{\hsize=7cm\noindent
{\sc B\'alint T\'oth}
\\
School of Mathematics
\\
University of Bristol
\\
Bristol, BS8 1TW
\\
United Kingdom
\\
email: {\tt balint.toth@bristol.ac.uk}
}
}


\begin{thebibliography}{99}

\bibitem{anders-deuschel-slowik_15}
S Anders, J-D Deuschel, M Slowik: 
Invariance principle for the random conductance model in a degenerate
ergodic environment. 
{\sl Ann. Probab.} {\bf 43:} 1866-1891 (2015)

\bibitem{barlow_04}
M Barlow:
Random walks on supercritical percolation clusters.
{\sl Ann. Probab.} {\bf 32:} 3024-3084 (2004)

\bibitem{bass_02}
RF Bass: 
On Aronson's upper bounds for heat kernels.
{\sl Bull. London Math. Soc.} {\bf 34:}415-419 (2002)


\bibitem{berger-biskup_07}
N Berger, M Biskup: 
Quenched invariance principle for simple random walk on percolation clusters. 
{\sl Probab. Theory Rel. Fields} {\bf 137:} 83-120 (2007)

\bibitem{biskup-prescott_07}
M Biskup, TM Prescott:
Functional CLT for random walk among bounded random conductances. 
{\sl Electr. Journ. Probab.} {\bf 12:} (paper no. 49) 1323-1348, (2007)

\bibitem{biskup_11}
M Biskup: 
Recent progress on the random conductance model.
{\sl Probab. Surveys} {\bf 8:} 294-373 (2011)

\bibitem{biskup-rodriguez_17}
M Biskup, P-F Rodriguez: 
Limit theory for random walks in degenerate time-dependent random environment. 
{\sl Preprint} (2017)
\href{https://arxiv.org/pdf/1703.02941.pdf}{\nolinkurl{arXiv:1703.02941}}

\bibitem{chacon-ornstein_60}
RV Chacon, DS Ornstein: 
A general ergodic theorem. 
{\sl Illinois J. Math.} {\bf 4}: 153-160 (1960)

\bibitem{hall-heyde_80}
P Hall, CC Heyde:
{\sl Martingale limit theory and its application}.
Academic Press, New York, 19080

\bibitem{helland_82}
IS Helland:
Central limit theorems for martingales with discrete or continuous time. 
{\sl Scand. J. Statist.} {\bf 9}: 79-94 (1982) 


\bibitem{hopf_54}
E Hopf:
The general temporally discrete Markoff process.
{\sl  J. Rational Mech. Anal.} {\bf 3} 13-45 (1954)

\bibitem{hopf_60}
E Hopf:
On the ergodic theorem for positive linear operators. 
{\sl J. Reine Angew. Math.} {\bf 205}: 101-106 (1960) 

\bibitem{horvath-toth-veto_12}
I Horv\'ath, B T\'oth, B Vet\H o:
Relaxed sector condition.
{\sl Bull. Inst. Math. Acad. Sin. (N.S.)} {\bf 7}: 463--476 (2012)

\bibitem{komorowski-landim-olla_12}
T Komorowski, C Landim, S Olla:
{\sl Fluctuations in Markov Processes -- Time Symmetry and Martingale Approximation}.
{\sl Grundlehren der mathematischen Wissenschaften}, Vol. 345,
Springer, Berlin-Heidelberg-New York, 2012

\bibitem{kozlov_85}
SM Kozlov:
The method of averaging and walks in inhomogeneous environments.
{\sl Uspekhi Mat. Nauk} {\bf 40}: 61--120 (1985)
English version:
{\sl Russian Math. Surveys} {\bf 40}: 73--145 (1985)

\bibitem{kozma-toth_14}
G Kozma, B T\'oth:
Central limit theorem for random walks in divergence-free random drift field: $H_{-1}$ suffices. 
\href{http://arxiv.org/abs/1411.4171v1}{\nolinkurl{arXiv:1411.4171}}

\bibitem{kozma-toth_17}
G Kozma, B T\'oth: 
Central limit theorem for random walks in doubly stochastic random environment: $H_{-1}$ suffices. 
{\sl Ann. Probab.} (to appear, 2017)
\href{http://arxiv.org/abs/1702.06905}{\nolinkurl{arXiv:1702.06905}}
 

\bibitem{krengel_85}
U Krengel:
{\sl Ergodic Theorems.}
De Gruyter Studies in Mathematics, vol 6.
De Gruyter, 1985
 

\bibitem{kumagai_14}
T Kumagai: 
Random Walks on Disordered Media and their Scaling Limits.
{\sl Lecture Notes in Mathematics}, Vol. 2101, {\sl \'Ecole d'\'Et\'e de Probabilit\'es de Saint-Flour XL--2010}. Springer, New York, (2014).

\bibitem{morris-peres_05}
B Morris, Y Peres: 
Evolving sets, mixing and heat kernel bounds. 
{\sl Probab. Theory Relat. Fields} {\bf 133:} 245-266 (2005)

\bibitem{nash_58}
J Nash:
Continuity of solutions of parabolic and elliptic equations.
{\sl Amer. Math. J.} {\bf 80:} 931-954 (1958)

\bibitem{sidoravicius-sznitman_04}
V Sidoravicius, A-S Sznitman:
Quenched invariance principles for walks on clusters of percolation or among random conductances. 
{\sl Probab. Theory Relat. Fields} {\bf 129}: 219-244 (2004)

\bibitem{zeitouni_04}
O Zeitouni:
Lecture notes on random walks in random environment.
In: \emph{Lectures on probability theory and statistics --- Saint-Flour 2001}.
Ed.: Jean Picard.
{\sl Lecture Notes in Mathematics} {\bf 1837}
Springer-Verlag, Berlin, 2004.

\end{thebibliography}
\end{document}